\documentclass[reqno]{amsart}
\usepackage[foot]{amsaddr}
\usepackage[english]{babel}
\AtBeginDocument{}
\usepackage[margin=1in,centering]{geometry}
\usepackage[utf8]{inputenc}
\usepackage{amsmath,amssymb,amsthm}
\numberwithin{equation}{section}
\usepackage{mathtools}
\usepackage{enumitem}
\usepackage{comment}
\usepackage{mleftright}

\usepackage{color}
\usepackage{hyperref}
\usepackage{cleveref}
\newcommand{\parens}[1]{\mleft(#1\mright)}
\newcommand{\set}[1]{\mleft\{#1\mright\}}

\newcommand{\spliteq}[2]{\begin{equation}#1\begin{split}#2\end{split}\end{equation}}
\newcommand{\eq}[1]{\begin{equation}#1\end{equation}}
    \newcommand{\blfootnote}[1]{{\renewcommand{\thefootnote}{\roman{footnote}}\footnotetext[0]{#1}}}
\newtheorem{theorem}{Theorem}
\newtheorem{lemma}{Lemma}
\newtheorem{corollary}{Corollary}

\renewcommand{\epsilon}{\varepsilon}

\begin{document}
\title{Zeros of Ramanujan-Like Polynomials
}

\author{Benjamin Fichter${}^{1\dag}$, Vishal Shah${}^{1\dag}$, Wiseley Wong${}^{1\ast}$}

\begin{abstract} 
In this paper we study a naturally occurring variant of Ramanujan polynomials with a Bernoulli convolution structure, $S_v$, related to the double Mordell-Tornheim series. In particular, we prove for even $v$ all zeros of $S_v$ lie on the complex unit circle and for even $v \ge 4$ are simple with the exception of roots at $1, -1$ which are double roots.

%trivial zeros of the double Mordell-Tornheim series at negative even integers imply that $|\alpha|=|\beta|$. We do so by analyzing a sequence of self-inversive polynomials with a Bernoulli convolution structure, $S_v(z)$. In particular, we prove all of $S_v(z)$ zeros lie on the complex unit circle. We use tools from Lal\'in and Smyth \cite{lalin2012unimodularityzerosselfinversivepolynomials} as well as a Bernoulli number convolution identity from Agoh and Dilcher \cite{AGOH2007105}. In addition, we note similarities to well studied Ramanujan polynomials \cite{maji2023zerosramanujantypepolynomials}.
\end{abstract}

\maketitle

\blfootnote{${}^{\dag}$ denotes joint first authorship, \email{bfichte1@umd.edu}, \email{shah2025@umd.edu}}
  \blfootnote{${}^{\ast}$ denotes corresponding author, \email{wwong123@umd.edu}}
\blfootnote{${}^1$~University of Maryland, College Park, MD 20742, USA}

% Other things to write about
% not product of cyclomatics follows from pidgon hole principle and lemma 2.
% Conj, Are their finitely meny common roots,
% speculative: set of roots are dense on unit circle?
% what do multiplicities tell us abt series?
% upper bound on |S_v| for z on unit circle
% is S_v/(z^2-1)^2 irreducable in Z[x]? - hard and req hard gaolios theory

\section{Introduction}

The motivation for this paper is to better predict the presence of trivial zeros of the analytic continuation of the double Mordell-Tornheim series initially studied by Tornheim \cite{tornheim1950harmonic}. The definition of the double Mordell–Tornheim series for arbitrary complex $\alpha$ and $\beta$ is 
\[
\sum_{m,n\geq1} \frac1{n^s}\frac1{m^s}\frac1{(\alpha n+\beta m)^s}.
\]

For our purposes, we assume $\alpha$ and $\beta$ are non-zero. We indirectly prove a conjecture proposed by Karl Dilcher \cite{PrivateCommunication}: \textit{If the double Mordell-Tornheim series has trivial zeros at some even negative integer $s$, then} $|\alpha|=|\beta|$. Karl Dilcher initially postulated the conjecture in the form of trivial zeros of the double Mordell-Tornheim series for $\alpha=1$, which is algebraically equivalent by choosing $\beta'= \beta/\alpha$ and multiplying through by $\frac1{\alpha^s}$ \cite{PrivateCommunication}.
There exists a naturally occurring sequence of polynomials $S_v(z)$, defined in Theorem \ref{thm1}, related to the double Mordell-Tornheim series using Crandall’s free-parameter method and Bernoulli-Barnes numbers \cite{PrivateCommunication,borwein2018derivatives}. In particular, the double Mordell-Tornheim series evaluated at negative even integers $-v$ is proportional to $S_v(\alpha/\beta)$. Hence, we prove the following statement equivalent to the conjecture above. We define $B_n$ as the $n$-th Bernoulli number.

\begin{theorem}\label{thm1}
{Polynomials of the form}
\eq{S_v(z) := 1+z^{3v+2}+2(2v+1)\binom{2v}{v}\frac{3v+2}{B_{3v+2}} \sum_{r=0}^{\left\lfloor   \frac{v}{2} \right\rfloor -1} \binom{v}{2r+1}\frac{B_{v+2+2r}}{v+2+2r}\cdot \frac{B_{2v-2r}}{2v-2r}z^{v+2+2r} 
}
have all of their zeros on the complex unit circle for all even non-negative $v$.
\end{theorem}

These polynomials $S_v$ and their relation to the double Mordell-Tornheim series have many similarities to Ramanujan polynomials \cite{dixit2024recent,maji2023zeros,smyth2011zeros}, defined as 
\[
    R_{2k+1}(z):=\sum_{j=0}^{k+1} \frac{B_{2j}B_{2k+2-2j}}{(2j)!(2k+2-2j)!}z^{2j}.
\]
For example, both have roots restricted to the unit circle, contain a convolution of Bernoulli numbers, and can be used for evaluation of number theoretic series at negative integers.
%Ramanujan polynomials share similarities with $S_v(z)$, their zeros being of significance to a zeta function and their construction relating to identities involving Bernoulli numbers. 
Murty, Smyth, and Wang \cite{smyth2011zeros} proved that all non-real roots of Ramanujan polynomials lie on $|z|=1$. This property also holds for $S_v(z)$, but unlike Ramanujan polynomials, this property is extended to real roots as well. Future work involving $S_v$ may be motivated by extending results known for Ramanujan polynomials.

\subsection{Tools and Preliminaries}\label{tools}
We will prove Theorem 1 by using a classical theorem of Lakatos and Losonczi \cite[Theorem 1]{lakatos2004self}:
\begin{theorem}\label{UnitCircle-thm}
    
\textit{If polynomial $P(z)=\sum_{j=0}^{d}A_jz^j$ is a self-inversive polynomial and}
\eq{\label{Inequality}
    \frac12\sum_{j=1}^{d-1} |A_j| \le |A_d|,
}
then all zeros of $P(z)$ lie on the unit circle.
\end{theorem}
For our application $P(z)=S_v(z)$ and $d=3v+2=6n+2$. Note that there is a brief proof of Theorem \ref{UnitCircle-thm} by Lal\'in and Smyth \cite[Corollary 3]{lalin2013unimodularity}. The proof of this corollary relies on Rouch\'es theorem \cite{Conway1978} applied to a family of \textit{self-inversive} polynomials, particularly functions that satisfy $P(z)=z^{d-n}h(z)+z^n\bar{h}(z)$. Interestingly, the motivation of Lal\'in and Smyth is to study the roots of Ramanujan polynomials. 

We break up the result into two sections, one to prove $S_v$ is self-inversive, and the other to prove inequality \eqref{Inequality} holds. We prove the following lemma in Section 2 to verify $S_v$ is self-inversive.

\begin{lemma}\label{lemma-inversive}
    We have for all even non-negative $v$ that $S_v$ is a self-inversive polynomial: There exists $\xi\ne 0$ such that
    \eq{\label{Inversive_eq}S_v(z) = \xi z^{3v+2}S_v(1/z).}
    In particular, $\xi=1$.
\end{lemma}
We recognize other  equivalent alternate definitions for a self-inversive polynomial. For example, the multiset of roots of $P(z)$ equal the multiset of roots for $z^dP(1/z)$ \cite{propselfinversive}. In Section 3 we continue by proving and manipulating the following identity in order to prove Theorem \ref{thm1}:

\begin{lemma}\label{lemma-identity}
For all $n\ge1$, we have
\eq{
\frac{B_{6n+2}}{6n+2} = -(4n+1)\binom{4n}{2n}\sum_{r=0}^{n-1}\binom{2n}{2r+1}\frac{B_{2r+2n+2}}{2r+2n+2} \frac{B_{4n-2r}}{4n-2r} .\quad
}
\end{lemma}
In Section 4 we further analyze the root structure of $S_v$ by showing $\pm1$ are the only double roots for $v\ge4$. Again we use tools from Lakatos and Losonczi, but first introduce some definitions using the standard notation for the argument of a complex number $z$, $\text{arg}(z)$ \cite[Theorem 1(ii)-2]{lakatos2004self}:
\begin{align*}
    \beta_j &= \text{arg}\left( A_j \parens{\frac{\Bar{A_0}}{A_{d}}}^{1/2}\right), \\
    \phi_j &= \frac{2(\pi j-\beta_j)}{d}.
\end{align*}

\begin{theorem}\label{Simple-Roots-Thm}
If all roots of $P(x)$ are on the unit circle and inequality \eqref{Inequality} holds with equality, then all roots of $P(z)$ are either simple or double roots. All double roots are of the form $e^{i\phi_j}$. In addition $e^{i\phi_j}$ is a double root of $P(z)$ if and only if for all $k\in \{1,2,\ldots,\lfloor\frac{\deg(P)}{2} \rfloor \}$ such that $A_k\ne0$,
    \eq{\label{cosineeq}
    \cos{\parens{\beta_{d-k}+\left(\frac{\deg(P)}{2}-k\right)\phi_j}} = (-1)^{j+1}.
    }
\end{theorem}

Finally, in Section \ref{sec:conclusion}, we discuss some possibilities for further work.

\section{Self Inverse Polynomial}

Here we prove Lemma \ref{lemma-inversive} and show that Equation \eqref{Inversive_eq} holds for all even $v \ge0$.

\begin{proof}[Proof of Lemma 1] We need to show $S_v(z)=z^{3v+2}\overline{S_v}(1/z)$ for even non-negative $v$ to show $S_v(z)$ is self-inverse. Observe that $\overline{S_v}=S_v$ since $S_v$ only has real coefficients. Taking $v=2n$, we have

\begin{align*}
z^{6n+2}\overline{S_{2n}}(1/z)=z^{6n+2}+1+
\frac{2(6n+2)(4n+1)\binom{4n}{2n}}{B_{6n+2}}\sum_{r=0}^{n-1}\binom{2n}{2r+1}\frac{B_{2n+2r+2}}{2n+2r+2}\frac{B_{4n-2r}}{4n-2r}z^{4n-2r}.
\end{align*}

Now, we just need to show that these polynomials are indeed equivalent. Recall that $A_j$ is the coefficient of $z^j$ in $S_v(z)$. Here, $A_0 = A_d = 1$ and so we just focus on all intermediate $j=2n+2,2n+4,\ldots,4n$. For such $j$, we have $A_{j}={2(4n+1)\binom{4n}{2n}\frac{6n+2}{B_{6n+2}}} \binom{2n}{4n-j+1}\frac{B_{d-j}}{d-j}\frac{B_{j}}{j}.$ We need to show that for all $j$, we have $A_j=A_{d-j}$. The coefficient of $z^{2n+2r+2}$ in $S_{2n}(z)$ is
\[
{2(4n+1)\binom{4n}{2n}\frac{6n+2}{B_{6n+2}}} \binom{2n}{2r+1}\frac{B_{2n+2r+2}}{2n+2r+2}\frac{B_{4n-2r}}{4n-2r}.
\]
We neglect the $2(4n+1)\binom{4n}{2n}\frac{6n+2}{B_{6n+2}}$ prefactor that all nontrivial powers share and focus on the $r$-dependent $\binom{2n}{2r+1}\frac{B_{2n+2r+2}}{2n+2r+2}\frac{B_{4n-2r}}{4n-2r}$.
%Both polynomials (excluding $z^0,z^{6n+2}$) have terms with powers $\{z^{2n+2},z^{2n+4},\ldots,z^{4n}\}$. We show in both polynomials, that these powers have the same coefficients.
%\\
%\\
Now note that to match powers, we must take $r'=n-1-r$ in the sum for $z^{6n+2}S_{2n}(1/z)$. Using  substitution, the coefficient is
\begin{align*}
&\binom{2n}{2(n-1-r)+1} \frac{B_{2n+2(n-1-r)+2}}{2n+2(n-1-r)+2}
\frac{B_{4n-2(n-1-r)}}{4n-2(n-1-r)}
\\&=
\binom{2n}{2n-2r-1} \frac{B_{4n-2r}}{4n-2r} \frac{B_{2n+2r+2}}{2n+2r+2}
\\&=
\binom{2n}{2r+1} 
\frac{B_{2n+2r+2}}{2n+2r+2} 
\frac{B_{4n-2r}}{4n-2r}.
\end{align*}

Clearly, $1$ and $z^{6n+2}$ have the same coefficient of $1$, and all powers with nontrivial coefficients have the same coefficients after the ``inversion". Hence, $S_{v}(z)$ is self-inversive for even $v$.
\end{proof}

\section{Proof of Theorem 1}
In this Section we use Theorem \ref{UnitCircle-thm}, that for some self-inversive polynomial with coefficients $\set{A_i}$ and degree $d$, the inequality
\[
|A_d| \geq \frac12 \sum_{j=1}^{d-1}\left|A_j \right|
\]
implies the self-inversive polynomial has roots exclusively on the unit circle. First, we prove Lemma \ref{lemma-identity}.

\begin{proof}[Proof of Lemma 2]
We begin with \cite[Theorem 2.1, Corollary 2.1 (Equation 2.2)]{agoh2007convolution} as shown in Section \ref{tools}. For integers $k \ge 2$ and $m \ge 0$:
\begin{equation} \label{agoh}
\begin{split}
(B_k + B_k)^m = & -\frac{(k!)^2}{(2k+1)!}(m+2k+1)B_{m+2k} \\
& + \frac{(-1)^{k+1}}{k+1}\sum_{r=0}^{\lfloor k/2 \rfloor} \frac{B_{2(k-r)}}{k-r}\binom{k+1}{2r+1}\big((k-2r)m - (2r+1)k\big)B_{m+2r}.
\end{split}
\end{equation}

Substitute $k=2n$ and $m=2n+2$ into both sides of Equation \eqref{agoh}.  Evaluating the left-hand side of Equation \eqref{agoh}, the umbral convolution expands to non-zero even indexed terms since odd indexed Bernoulli numbers (with index greater than or equal to 2) are 0. The indices of the Bernoulli numbers above are all larger than $2n+2$ for $n \geq 0$, thus $B_1$ does not appear in the sum. For more on umbral calculus see \cite{gessel2001applications}. This allows us to set $j=2s$:

\begin{align*}
(B_{2n}+B_{2n})^{2n+2} &=  \sum_{j=0} ^{2n+2}\binom{2n+2}{j}B_{2n+j}B_{4n+2-j}\\
&= \sum_{s=0}^{n+1} \binom{2n+2}{2s} B_{2n+2s} B_{4n+2-2s} \\
&= 2B_{2n}B_{4n+2} + \sum_{s=1}^n \binom{2n+2}{2s} B_{2n+2s} B_{4n+2-2s}.
\end{align*}
Shifting the summation index by $r = s-1$ yields the left-hand side as
\begin{equation} \label{lhs}
2B_{2n}B_{4n+2} + \sum_{r=0}^{n-1} \binom{2n+2}{2r+2} B_{2n+2r+2} B_{4n-2r}.
\end{equation}
Now, we look at the right-hand side of Equation \eqref{agoh}. We again let $k=2n, \: m=2n+2$, so that this right-hand side becomes
\begin{equation} 
\begin{split}
-\frac{((2n)!)^2}{(4n+1)!}(6n+3)B_{6n+2} -\frac{1}{2n+1}\sum_{r=0}^{n}\binom{2n+1}{2r+1}\frac{B_{4n-2r}}{2n-r} \big((2n-2r)(2n+2)-2n(2r+1)\big)B_{2n+2r+2}.
\end{split}
\end{equation}
Since $(2n-2r)(2n+2)-2n(2r+1) = 2(n-2r)(2n+1)$, we have
\begin{align*}
    &-\frac{3(2n+1)}{(4n+1)\binom{4n}{2n}}B_{6n+2} - 2\sum_{r=0}^{n} \frac{n-2r}{2n-r}\binom{2n+1}{2r+1} B_{2n+2r+2}B_{4n-2r} \\
&= -\frac{3(2n+1)}{(4n+1)\binom{4n}{2n}}B_{6n+2} + 2B_{2n}B_{4n+2} - 2\sum_{r=0}^{n-1} \frac{n-2r}{2n-r}\binom{2n+1}{2r+1} B_{2n+2r+2}B_{4n-2r}.
\end{align*}
Thus, by Equation \eqref{agoh}, we have
\begin{align*}
    \sum_{r=0}^{n-1} \binom{2n+2}{2r+2} B_{2n+2r+2} B_{4n-2r}=-\frac{3(2n+1)}{(4n+1)\binom{4n}{2n}}B_{6n+2} - 2\sum_{r=0}^{n-1} \frac{n-2r}{2n-r}\binom{2n+1}{2r+1} B_{2n+2r+2}B_{4n-2r}.
\end{align*}
Solving for $B_{6n+2}$ and reducing yields
\begin{align*}
B_{6n+2} &= -\frac{(4n+1)}{3(2n+1)} \binom{4n}{2n} \sum_{r=0}^{n-1}\left(\binom{2n+2}{2r+2}+2\cdot\frac{n-2r}{2n-r} \binom{2n+1}{2r+1}\right)B_{2n+2r+2}B_{4n-2r} \\
&= -\frac{(4n+1)}{3(2n+1)} \binom{4n}{2n} \sum_{r=0}^{n-1}\binom{2n}{2r+1} {\left[ \frac{(2n+1)(n+1)}{(r+1)(2n-2r)} + \frac{2(n-2r)(2n+1)}{(2n-r)(2n-2r)}\right]}B_{2n+2r+2}B_{4n-2r}.
\end{align*}
We define
\[
X_r:= \frac{(2n+1)(n+1)}{(r+1)(2n-2r)} + \frac{2(n-2r)(2n+1)}{(2n-r)(2n-2r)}.
\]
The binomial $\binom{2n}{2r+1}$ and the Bernoulli product $B_{2n+2r+2}B_{4n-2r}$ are invariant under the reflection $r \to n-1-r$, so we symmetrize the sum and find 
\eq{B_{6n+2} = -\frac{(4n+1)}{3(2n+1)} \binom{4n}{2n} \sum_{r=0}^{n-1}\binom{2n}{2r+1} \left( \frac{X_r+X_{n-1-r}}{2}\right) B_{2n+2r+2}B_{4n-2r}. }
%Thus, giving equality to the sum over $X_r$ symmetrized average $\frac{1}{2}(X_r + X_{n-1-r})$. 
This simplifies nicely as
\[
\frac{1}{2}(X_r + X_{n-1-r}) = \frac{3(2n+1)(3n+1)}{2(2n-r)(n+r+1)},
\]
so that
%Substituting $\frac{1}{2}(X_r + X_{n-1-r})$ into the identity yields
\begin{align*}
B_{6n+2} &= -\frac{4n+1}{3(2n+1)} \binom{4n}{2n} \sum_{r=0}^{n-1}\binom{2n}{2r+1} \left[ \frac{3(2n+1)(3n+1)}{2(2n-r)(n+r+1)} \right] B_{2n+2r+2}B_{4n-2r}.
\end{align*}
By dividing by $6n+2$ and simplifying, we have the identity in our desired form,
\eq{
\frac{B_{6n+2}}{6n+2} = -(4n+1)\binom{4n}{2n}\sum_{r=0}^{n-1}\binom{2n}{2r+1}\frac{B_{2r+2n+2}}{2r+2n+2} \frac{B_{4n-2r}}{4n-2r}.}
\end{proof}
Now we are ready to prove Theorem \ref{thm1}.
\begin{proof}[Proof of Theorem 1]
\noindent From Lemma \ref{lemma-identity},
\[
\frac{B_{6n+2}}{6n+2} = -(4n+1)\binom{4n}{2n}\sum_{r=0}^{n-1}\binom{2n}{2r+1}\frac{B_{2r+2n+2}}{2r+2n+2} \frac{B_{4n-2r}}{4n-2r}.
\]
Equivalently,
\begin{equation}
\label{NonAbsSumofA}
\begin{split}
1   &=-\frac{(6n+2)(4n+1)\binom{4n}{2n}}{B_{6n+2}}\sum_{r=0}^{n-1}\binom{2n}{2r+1}\frac{B_{2n+2r+2}}{2n+2r+2}\frac{B_{4n-2r}}{4n-2r}
\\  &=-\frac12 \sum_{r=0}^{n-1}\frac{2(6n+2)(4n+1)\binom{4n}{2n}}{B_{6n+2}}\binom{2n}{2r+1}\frac{B_{2n+2r+2}}{2n+2r+2}\frac{B_{4n-2r}}{4n-2r}
\\  &=-\frac12 \sum_{r=0}^{n-1}A_{2n+2r+2}
\\  &=-\frac12 \sum_{j=1}^{d-1}A_{j}.
\end{split}
\end{equation}

We use the standard notation for the sign of $x$, $\text{sgn}(x)$ to help us find the sign of each term. It is well known that $\text{sgn}(B_{2m})=(-1)^{m+1}$. So, $\text{sgn}(B_{2n+2r+2})\cdot\text{sgn}(B_{4n-2r})\cdot \text{sgn}(B_{6n+2})=(-1)^{n+r+2}\cdot (-1)^{2n-r+1}\cdot (-1)^{3n+2}=-1$. Hence, for $A_j\ne0$ and $j\ne0,d$,
\eq{\label{sgnAj}\text{sgn}(A_j)=-1.}
This implies $|A_j| = -A_j$. Substituting this into Equation \eqref{NonAbsSumofA} yields the equation
\[
1 = \frac12\sum_{j=1}^{d-1}|A_j|.
\]
Since $|A_0|=|A_d|=1$,
\eq{
|A_d|=\frac12\sum_{j=1}^{d-1}|A_j|.
}

\noindent Thus, inequality \eqref{Inequality} holds with equality. Therefore, by Lemma \ref{lemma-inversive} and Theorem \ref{UnitCircle-thm}, all zeros of $S_v$ lie on the unit circle.
\end{proof}

\section{Double Roots}
In this Section we prove consequences of equality from Section 3. 
In particular, the only roots of $S_v$ with multiplicity two are $\pm1$. We use Theorem \ref{Simple-Roots-Thm} from Lakatos and Losonczi \cite{lakatos2004self}. 

\begin{corollary}
    For all even $v\ge4$, $S_v(z)$ have double roots at $\pm1$ and all other roots are simple.
\end{corollary}

\begin{proof}
Observe some simplifications of $\beta_j$ and $\phi_j$. We have $A_j=A_{d-j}$ as $S_v(z)$ is \textit{self-inversive} with all real coefficients. Hence, we can simplify $\beta_j$:
\begin{align*}
\beta_j&=\text{arg}\left(A_j\left(\frac{\Bar{A_0}}{A_{d}}\right)^{1/2}\right) \\
    &=\text{arg}\left(A_j\right)
\end{align*} 

By using Equation \ref{sgnAj}, we have $\beta_j= \pi$ when $A_j\ne0$. Since we know the degree of $S_v$ we can substitute $d=6n+2$. Then $\phi_j= \frac{\pi j}{3n+1}$. Now we find all $j$ such that for all even $k$ satisfying $2n+2\le k \le 3n+1$.
\spliteq{}{
(-1)^{j+1} 
&=\cos{\parens{\beta_{6n+2-k}+(3n+1-k)\phi_j}} \\
&= \cos{\parens{\pi+(3n+1-k)\frac{\pi j}{3n+1}}}\\
&= \cos{\parens{\pi (j+1)-\frac{\pi jk}{3n+1}}}
}
Therefore,
\eq{
    \pi (j+1)-\frac{\pi jk}{3n+1} \in 2\pi \mathbb{Z}+\pi(j+1).
}
Simplifying yields,
\eq{
    \frac{jk}{3n+1} \in 2\mathbb{Z}.
}
Since $k$ is even, we can replace $k=2k'$. Thus for all integers $k'$, such that $n+1\le k' \le \frac{3n+1}{2} $,
\eq{\label{EndingSection4}
    \frac{jk'}{3n+1} \in \mathbb{Z}.
}
We now show $3n+1 | j$ must hold for Equation \eqref{EndingSection4} to hold by splitting into two cases, $n=2$ and $n\ge 3$. First, for $n\ge 3$$n+1,n+2 \le (3n+1)/2$, thus $n+1,n+2$ are values of $k'$ for which the above must holds. Since $k'=n+1,n+2$ are necessarily coprime, in order for Equation \eqref{EndingSection4} to hold, $3n+1 | j$. Now we consider $n=2$, necessarily $k'=3$, and $3n+1=7$, implying $k',3n+1$ co-prime, thus $3n+1 | j$. Therefore, for $n\ge 2$ Equation \eqref{EndingSection4} holds if and only if $j$ is a multiple of $3n+1$. By Theorem \ref{Simple-Roots-Thm}, $e^{i\phi_j}$ is a double root if and only if $j=0$ or $3n+1$. Thus, by evaluating $e^{i\phi_j}$, the only double roots of $S_v(z)$ are $\pm1$, and all other roots of $S_v(z)$ are simple and on the unit circle.
\end{proof}
\section{Conclusion}\label{sec:conclusion}
%This proof reducing to a Bernoulli convolution identity is of interest. 
%We can expect other Ramanujan-like polynomials with binomial coefficients to have similar properties for their roots. We can also expect this result to hold for other formulations involving binomial coefficients. In particular other connections between specific trivial zeros of a zeta-inspired series, a sequence of polynomials and a general identities involving Bernoulli numbers. %However methods for finding a series associated with an identity appear unexplored.

We can expect other Ramanujan-type polynomials involving binomial coefficients and Bernoulli number coefficients to have similar properties regarding their roots. It would be interesting to systematically evaluate which related polynomials also have all roots on the complex unit circle, especially when they naturally arise as counterparts of a zeta-type series. 

%Explicitly, finding existing or new Bernoulli identities which determine the coefficients of some family of polynomials with the property $|A_d|=\frac12 \sum _{j=1}^{d-1}|A_j|$. 

% However, finding a series with properties that match the construction is not as clear.

\section{Acknowledgments}
We thank Tanay Wakhare for providing valuable feedback. GPT 5.4-Pro was used during the ideation phase. All claims were verified by the authors, who take full responsibility for the final mathematical accuracy of this work.
\nocite{*}
%\pagebreak
\bibliographystyle{plain}
\bibliography{refs}

\end{document}